\documentclass[preprint,3p,times]{elsarticle}

\usepackage{amssymb}
\usepackage{amsmath}
\usepackage[all,cmtip]{xy}
\usepackage{amsthm}
\usepackage{color}
\usepackage{enumitem,verbatim}
\usepackage{tikz}
\usepackage{appendix}
\usepackage[colorlinks=true]{hyperref}
\usepackage[hyphenbreaks]{breakurl}
\input diagxy

\theoremstyle{plain}
  \newtheorem{thm}{Theorem}[section]
  \newtheorem{lem}[thm]{Lemma}
  \newtheorem{prop}[thm]{Proposition}
  \newtheorem{cor}[thm]{Corollary}
\theoremstyle{definition}
  \newtheorem{defn}[thm]{Definition}
  \newtheorem{exmp}[thm]{Example}
  \newtheorem{rem}[thm]{Remark}

\newtheorem*{SA}{Standing Assumption}
\newtheorem*{Discus}{Discussion}

\DeclareMathAlphabet{\mathcal}{OMS}{cmsy}{m}{n}

\DeclareMathOperator{\id}{id}
\DeclareMathOperator{\ob}{ob}

\DeclareMathOperator{\Int}{Int}

\makeatletter
\def\ps@pprintTitle{%
 \let\@oddhead\@empty
  \let\@evenhead\@empty
  \def\@oddfoot{\vbox{\hsize=\textwidth\footnotesize
  \vskip 8pt
  \copyright 2020. This manuscript version is made available under the CC-BY-NC-ND 4.0 license \url{https://creativecommons.org/licenses/by-nc-nd/4.0/}. The published version is available at \url{https://doi.org/10.1016/j.fss.2020.01.013}.\\
  }}%
  \let\@evenfoot\@oddfoot}
\makeatother

\newcommand{\da}{\downarrow}
\newcommand{\ua}{\uparrow}
\newcommand{\ra}{\rightarrow}

\newcommand{\lda}{\swarrow}
\newcommand{\rda}{\searrow}
\newcommand{\bs}{\backslash}

\newcommand{\bv}{\bigvee}
\newcommand{\bw}{\bigwedge}

\newcommand{\dv}{\dashv}

\newcommand{\rqa}{\rightsquigarrow}
\newcommand{\lar}{\looparrowright}

\renewcommand{\phi}{\varphi}
\newcommand{\al}{\alpha}
\newcommand{\be}{\beta}

\newcommand{\ga}{\gamma}

\newcommand{\Om}{\Omega}

\newcommand{\CC}{\mathcal{C}}

\newcommand{\CI}{\mathcal{I}}

\newcommand{\CO}{\mathcal{O}}

\newcommand{\CQ}{\mathcal{Q}}
\newcommand{\CR}{\mathcal{R}}

\newcommand{\sB}{{\sf B}}

\newcommand{\sQ}{{\sf Q}}
\newcommand{\sR}{{\sf R}}

\newcommand{\sj}{{\sf j}}

\newcommand{\BB}{{\bf B}}
\newcommand{\BD}{{\bf D}}

\newcommand{\Cat}{{\bf Cat}}

\newcommand{\Rel}{{\bf Rel}}

\newcommand{\Sup}{{\bf Sup}}

\newcommand{\QCat}{\CQ\text{-}\Cat}
\newcommand{\RCat}{\CR\text{-}\Cat}

\newcommand{\op}{{\rm op}}

\newcommand{\BsQ}{\BB(\sQ)}
\newcommand{\KsQ}{\BB_*(\sQ)}

\newcommand{\KC}{\BB_*(C_3)}

\newcommand{\DsQ}{\BD(\sQ)}
\newcommand{\HsQ}{\BD_*(\sQ)}
\newcommand{\HssQ}{\BD_*(\sQ)^{\circ}}
\newcommand{\HC}{\BD_*(C_3)}

\newcommand{\OX}{\CO(X)}

\newcommand{\PCX}{{\sf PC}(X)}
\newcommand{\ldd}{\mathrel{/}}
\newcommand{\rdd}{\mathrel{\bs}}
\newcommand{\with}{\mathrel{\&}}

\renewcommand{\leq}{\leqslant}
\renewcommand{\geq}{\geqslant}

\numberwithin{equation}{section}

\allowdisplaybreaks

\begin{document}

\begin{frontmatter}



\title{Quantale-valued dissimilarity}


\author[S]{Hongliang Lai}
\ead{hllai@scu.edu.cn}

\author[S]{Lili Shen\corref{cor}}
\ead{shenlili@scu.edu.cn}

\author[P]{Yuanye Tao}
\ead{tyymath@foxmail.com}

\author[S]{Dexue Zhang}
\ead{dxzhang@scu.edu.cn}

\cortext[cor]{Corresponding author.}
\address[S]{School of Mathematics, Sichuan University, Chengdu 610064, China}
\address[P]{Ping An Bank Co., Ltd., Shenzhen 518001, China}

\begin{abstract}
Inspired by the theory of apartness relations of Scott, we establish a positive theory of dissimilarity valued in an involutive quantale $\mathsf{Q}$ without the aid of negation. It is demonstrated that a set equipped with a $\mathsf{Q}$-valued dissimilarity is precisely a symmetric category enriched in a subquantaloid of the quantaloid of back diagonals of $\mathsf{Q}$. Interactions between $\mathsf{Q}$-valued dissimilarities and $\mathsf{Q}$-valued similarities (which are equivalent to $\mathsf{Q}$-valued equalities in the sense of H{\"o}hle--Kubiak) are investigated with the help of lax functors. In particular, it is shown that similarities and dissimilarities are interdefinable if $\mathsf{Q}$ is a Girard quantale with a hermitian and cyclic dualizing element.
\end{abstract}


\end{frontmatter}




\section{Introduction}

In order to explain the motivation and the purpose of this paper, we start with a well-known example proposed by Fourman--Scott \cite{Fourman1979}. Given a topological space $X$, let $\OX$ be the frame of open sets of $X$, and let
\begin{equation} \label{PCX-def}
\PCX=\{f\mid f\ \text{is a real-valued continuous map on an open subset}\ D(f):=U\subseteq X\}.
\end{equation}
For any $f,g\in\PCX$, the value
\begin{equation} \label{sim-PCX}
\al(f,g):=\Int\{x\in D(f)\cap D(g)\mid f(x)=g(x)\},
\end{equation}
i.e., the interior of the set
\[\{x\in D(f)\cap D(g)\mid f(x)=g(x)\}\]
in $X$, may be treated as the truth-value, computed in the frame $\OX$, of the statement that \emph{$f$ is equal to $g$}. The pair $(\PCX,\al)$ is a prototype of \emph{frame-valued sets} in the sense of Fourman--Scott \cite{Fourman1979} and Higgs \cite{Higgs1970,Higgs1984}. Explicitly, considering a frame $\Om$ as the table of truth-values, an \emph{$\Om$-set} is a set $A$ that comes equipped with a map
\[\al:A\times A\to\Om\]
such that
\begin{itemize}
\item (symmetry)  $\al(x,y)=\al(y,x)$,
\item (transitivity)  $\al(y,z)\wedge\al(x,y)\leq\al(x,z)$
\end{itemize}
for all $x,y,z\in A$, where $\al(x,y)$ is interpreted as the truth-value that $x$ is \emph{similar} (or \emph{equal}, or \emph{equivalent}) to $y$. It is well known that the category of  $\Om$-sets is equivalent to the topos ${\sf Sh}(\Om)$ of sheaves over $\Om$ \cite{Fourman1979}.

As a dualization of \eqref{sim-PCX}, it is natural to consider the value
\begin{equation} \label{dissim-PCX}
\be(f,g):=\Int(X-\Int\{x\in D(f)\cap D(g)\mid f(x)=g(x)\})
\end{equation}
for any $f,g\in \PCX$, i.e., the interior of the complement of the interior of the set
\[\{x\in D(f)\cap D(g)\mid f(x)=g(x)\}\]
in $X$, as the truth-value of the statement that $f$ is \emph{dissimilar} (or \emph{unequal}, or \emph{inequivalent}) to $g$ (also computed in the frame $\OX$). In other words, $\be$ may be thought of as an $\OX$-valued \emph{dissimilarity} on the set $\PCX$.

In classical logic, with the law of double negation in our arsenal, a dissimilarity (or inequivalence) relation on a set may be postulated as the complement (or negation) of a similarity (or equivalence) relation; that is, similarity and dissimilarity are interdefinable in classical logic. However, in a non-classical logic, e.g., intuitionistic logic and many-valued logic, the law of double negation may fail, and thus similarity and dissimilarity may not be deduced from each other via negation. In the 1970s, Scott \cite{Scott1979} pointed out that an independent \emph{positive} theory of inequalities is required in intuitionistic logic. To achieve this, he established the theory of \emph{apartness relations}.

During the past decades, different approaches have been adopted in the search for a reasonable definition of dissimilarities in the many-valued setting; we refer to \cite{Cross2002,Dubois1980} for an overview. Some typical approaches are listed below, which are all defined through some variations of similarity:
\begin{itemize}
\item A dissimilarity relation is assumed to be the ``fuzzy complement'' (also called ``inverse'') of a similarity relation; see, e.g., \cite{Nakamura1982,Zwick1987}.
\item A dissimilarity between fuzzy sets is postulated as a similarity between their ``fuzzy complements'' \cite{Dubois1982c}.
\item A dissimilarity is defined as an analogue of the distance in a metric space \cite{Restle1959,Ruspini1973}.
\end{itemize}

The aim of this paper is to establish a \emph{positive} theory of dissimilarity valued in an \emph{involutive quantale} \cite{Mulvey1992}
\[\sQ=(\sQ,\with,k,{}^{\circ});\]
that is, our notion of dissimilarity will not be postulated as a ``negation'' or a ``complement'' of that of similarity. It should be noted that our notion of \emph{$\sQ$-valued dissimilarity}, in spite of being motivated by the theory of Scott \cite{Scott1979}, is conceptually different from his notion of apartness relation. In particular, the map $\be$ given by \eqref{dissim-PCX} is an $\OX$-valued dissimilarity on $\PCX$, but in general it is not an $\OX$-valued apartness relation on $\PCX$.

The paper is structured as follows. Section \ref{Quantales} reviews some basic notions about involutive quantales. Section \ref{Sim-Dissim-Def} presents the key notion of this paper, i.e., that of \emph{$\sQ$-valued dissimilarity}. $\sQ$-valued dissimilarity is a ``dualization'' of \emph{$\sQ$-valued similarity}, which originates from a series of works of H{\"o}hle and his collaborators \cite{Hoehle1992,Hoehle1995b,Hoehle1998,Hoehle2005,Hoehle1991}. Semantic meanings of the axioms of these two notions are analyzed in this section.

It is known from \cite{Hoehle2011a} that a set equipped with a $\sQ$-valued similarity is exactly a \emph{symmetric} category enriched in the quantaloid $\HsQ$, which is a subquantaloid of the quantaloid $\DsQ$ of \emph{diagonals} of $\sQ$ \cite{Hoehle2011a,Pu2012,Stubbe2014}. Section \ref{Sim-Dissim-Cat} shows that a dual conclusion holds for $\sQ$-valued similarities. Explicitly, it is demonstrated that a set equipped with a $\sQ$-valued dissimilarity is precisely a \emph{symmetric} category enriched in $\KsQ$, where $\KsQ$ is a subquantaloid of the quantaloid $\BsQ$ of \emph{back diagonals} of $\sQ$ \cite{Shen2016a}. Therefore, $\sQ$-valued similarity and $\sQ$-valued dissimilarity are both instances of the thesis of Lawvere \cite{Lawvere1973} that \emph{fundamental structures are themselves categories}.

Based on the categorical perspective, in Section \ref{Sim-vs-Dissim} we investigate the connections between $\sQ$-valued similarities and $\sQ$-valued dissimilarities by constructing \emph{lax functors} between the quantaloids $\HsQ$ and $\KsQ$, which are deeply affected by the structure of the quantale $\sQ$:
\begin{itemize}
\item If $\sQ$ is a divisible quantale with the bottom element being cyclic, then the negations of $\sQ$-valued dissimilarities are $\sQ$-valued similarities (Proposition \ref{divisible-neg-dissim}).
\item If $\sQ$ is a frame, then the negations of $\sQ$-valued dissimilarities are $\sQ$-valued similarities, and vice versa (Proposition \ref{frame-neg-sim-dissim}).
\end{itemize}
Furthermore, we confirm the intuition that similarity and dissimilarity are interdefinable when $\sQ$ satisfies the law of double negation. Explicitly, if $\sQ$ is a Girard quantale \cite{Rosenthal1990,Yetter1990}, then we have isomorphisms
\[\DsQ\cong\BsQ\quad\text{and}\quad\HsQ\cong\KsQ\]
of quantaloids (Theorem \ref{Girard-iso}); and moreover, if $\sQ$ is an involutive Girard quantale with a hermitian and cyclic dualizing element, then $\sQ$-valued similarity and $\sQ$-valued dissimilarity are fully decidable by each other (Theorem \ref{Girard-sim-dissim}). Conversely, for a commutative quantale $\sQ$, it is shown in Theorem \ref{DQ-BQ-Girard} that the existence of an isomorphism $\DsQ\cong\BsQ$ of quantaloids necessarily forces $\sQ$ to be a Girard quantale; hence, if $\sQ$ is commutative and integral, then
\[\HsQ\cong\KsQ\iff\sQ\ \text{is a Girard quantale},\]
which is recorded as Corollary \ref{DQ-BQ-Girard-integral}.

\section{Quantales} \label{Quantales}

A \emph{(unital) quantale} \cite{Mulvey1986,Rosenthal1990}
\[\sQ=(\sQ,\with,k)\]
is a monoid with $k$ being the unit, such that the underlying set $\sQ$ is a complete lattice (with a top element $\top$ and a bottom element $\bot$) and the multiplication $\with$ distributes over arbitrary suprema, i.e.,
\[p\with\Big(\bv_{i\in I}q_i\Big)=\bv_{i\in I}p\with q_i\quad\text{and}\quad\Big(\bv_{i\in I}p_i\Big)\with q=\bv_{i\in I}p_i\with q\]
for all $p,q,p_i,q_i\in\sQ$ $(i\in I)$. The induced right adjoints
\[(-\with q)\dv(-\ldd q):\ \sQ\to\sQ\quad\text{and}\quad(p\with -)\dv(p\rdd -):\ \sQ\to\sQ,\]
called \emph{left} and \emph{right implications} in $\sQ$, are given by
\begin{equation} \label{imp-def}
r\ldd q=\bv\{p'\in\sQ\mid p'\with q\leq r\}\quad\text{and}\quad p\rdd r=\bv\{q'\in\sQ\mid p\with q'\leq r\},
\end{equation}
respectively, which satisfy
\[p\with q\leq r\iff p\leq r\ldd q\iff q\leq p\rdd r\]
for all $p,q,r\in\sQ$. We say that
\begin{itemize}
\item $\sQ$ is \emph{commutative}, if $p\with q=q\with p$ for all $p,q\in\sQ$, in which case we write
    \[p\ra q:=q\ldd p=p\rdd q\]
    for all $p,q\in\sQ$;
\item $\sQ$ is \emph{integral}, if the unit $k=\top$, the top element of the complete lattice $\sQ$;
\item $\sQ$ is \emph{divisible}, if
    \begin{equation} \label{divisible-def}
    (u\ldd q)\with q=u=q\with(q\rdd u)
    \end{equation}
    whenever $u\leq q$ in $\sQ$, in which case $\sQ$ is necessarily integral. 
\item $\sQ$ is a \emph{complete MV-algebra} \cite{Chang1958}, if $\sQ$ is commutative and
    \begin{equation} \label{MV-def}
    (p\ra q)\ra q=p\vee q
    \end{equation}
    for all $p,q\in\sQ$, in which case $\sQ$ is necessarily divisible (cf. \cite[Lemma 2.5]{Galatos2005}).
\item $\sQ$ is \emph{involutive} \cite{Mulvey1992}, if there exists an \emph{involution} on $\sQ$; that is, a map $(-)^{\circ}:\sQ\to\sQ$ such that
    \[k^{\circ}=k,\quad q^{\circ\circ}=q,\quad (p\with q)^{\circ} = q^{\circ}\with p^{\circ}\quad\text{and}\quad \Big(\bv_{i\in I}q_i\Big)^{\circ}=\bv_{i\in I}q_i^{\circ}\]
    for all $p,q,q_i\in\sQ$ $(i\in I)$. In this case,
    \begin{itemize}
    \item it is easy to verify that
        \begin{equation} \label{invo-imp}
        (p\ldd q)^{\circ}= q^{\circ}\rdd p^{\circ}
        \end{equation}
        for all $p,q\in\sQ$;
    \item an element $q\in\sQ$ is called \emph{hermitian} (also \emph{self-adjoint}) if $q^{\circ}=q$, and $k$, $\top$, $\bot$ are clearly hermitian.
    \end{itemize}
\end{itemize}

\begin{exmp} \label{quantale-exmp}
We list here some quantales that are of concern in this paper:
\begin{enumerate}[label=(\arabic*)]
\item \label{quantale-exmp:Lawvere} Lawvere's quantale $[0,\infty]=([0,\infty],+,0)$ \cite{Lawvere1973} is commutative and divisible, where $[0,\infty]$ is the extended non-negative real line equipped with the order ``$\geq$'' (so that $0$ becomes the top element and $\infty$ the bottom element), and ``$+$'' is the usual addition extended via
    \[p+\infty=\infty+p=\infty\]
    to $[0,\infty]$, with $0$ being the unit and making $[0,\infty]$ a commutative and integral quantale. The implication in $[0,\infty]$ is given by
    \[p\ra q=\begin{cases}
    q-p & \text{if}\ p<q,\\
    0 & \text{else}
    \end{cases}\]
    for all $p,q\in[0,\infty]$, where the subtraction ``$-$'' is extended via
    \[\infty-p=\begin{cases}
    \infty & \text{if}\ p<\infty,\\
    0 & \text{if}\ p=\infty
    \end{cases}\]
    to $[0,\infty]$.
\item \label{quantale-exmp:frame} Every \emph{frame} $\Om=(\Om,\wedge,\top)$ is a commutative, divisible and idempotent quantale, and vice versa. In particular, the two-element Boolean algebra, denoted by ${\bf 2}$, is a frame. Moreover, each topological space $X$ gives rise to the frame $\OX=(\OX,\cap,X)$ of open sets of $X$.
\item \label{quantale-exmp:t-norm} Every \emph{complete BL-algebra} \cite{Hajek1998} is a commutative and divisible quantale. In particular, the unit interval $[0,1]$ equipped with a \emph{continuous t-norm} \cite{Klement2000} is a commutative and divisible quantale.
\item \label{quantale-exmp:nil-min} The unit interval $[0,1]$ equipped with the \emph{nilpotent minimum t-norm} \cite{Klement2000} is a commutative, integral and non-divisible quantale.
\item \label{quantale-exmp:C3} The three-chain $C_3=\{\bot,k,\top\}$ is equipped with a commutative and non-integral quantale structure $(C_3,\with,k)$, with
    \[\top\with\top=\top\ra\top=\top,\quad\top\ra\bot=\top\ra k=\bot\]
    and the other multiplications\,/\,implications being trivial.
\item \label{quantale-exmp:Rel} Let $\Rel(X)$ denote the set of (binary) relations on a non-empty set $X$. Then $(\Rel(X),\circ,\id_X)$ is an involutive quantale, where $\circ$ refers to the composition of relations, and
    \[\id_X=\{(x,x)\mid x\in X\}\]
    is the identity relation on $X$. It is obvious that the \emph{opposite} $R^{\circ}$ of relations $R\in\Rel(X)$, i.e.,
    \[R^{\circ}=\{(y,x)\in X\times X\mid(x,y)\in R\},\]
    defines an involution on $\Rel(X)$. Note that $\Rel(X)$ is non-commutative and non-integral as long as $X$ contains at least two elements.
\item \label{quantale-exmp:Sup} Let $\Sup[0,1]$ denote the set of $\sup$-preserving maps on the unit interval $[0,1]$. Then $(\Sup[0,1],\circ,1_{[0,1]})$ is a non-commutative, non-integral and unital quantale, where $\circ$ refers to the composition of maps, and $1_{[0,1]}$ is the identity map on $[0,1]$. An involution on $\Sup[0,1]$ is given by
    \[f^{\circ}:[0,1]\to[0,1],\quad f^{\circ}(x)=1-f^{\star}(1-x)\]
    for all $f\in\Sup[0,1]$, where $f^{\star}:[0,1]\to[0,1]$ is the right adjoint of $f$.
\item \label{quantale-exmp:c-inv} Every commutative quantale $\sQ$ is involutive, with a trivial involution given by the identity map on $\sQ$. In particular, all the commutative quantales mentioned in \ref{quantale-exmp:Lawvere}--\ref{quantale-exmp:C3} are involutive.
\end{enumerate}
\end{exmp}

\begin{SA}
Throughout this paper, we fix an involutive quantale
\[\sQ=(\sQ,\with,k,{}^{\circ})\]
as the table of truth-values, unless otherwise specified.
\end{SA}

\section{Quantale-valued similarity and dissimilarity: Definitions and examples} \label{Sim-Dissim-Def}

In order to throw light on the postulation of dissimilarity, let us recall the notion of \emph{$\sQ$-valued similarity}\footnote{A map $\al:X\times X\to\sQ$ is a $\sQ$-valued similarity in the sense of Definition \ref{sim-def} if, and only if,
\[\overline{\alpha}(x,y):=\al(y,x)\]
defines a $\sQ$-valued equality in the sense of H\"{o}hle-Kubiak (see \cite[Definition 2.1]{Hoehle2011a}). So, $\sQ$-valued similarity and $\sQ$-valued equality are equivalent concepts.} :

\begin{defn} \label{sim-def} (cf. \cite[Definition 2.1 and Lemma 2.3]{Hoehle2011a}.)
A \emph{$\sQ$-valued similarity} on a set $X$ is a map
\[\al:X\times X\to\sQ\]
such that
\begin{enumerate}[label=(S\arabic*)]
\item \label{sim-def:str} (strictness) \ $\al(x,y)\leq\al(x,x)\wedge\al(y,y)$,
\item \label{sim-def:sym} (symmetry) \ $\al(x,y)=\al(y,x)^{\circ}$,
\item \label{sim-def:div} (divisibility) \ $\al(x,y)=(\al(x,y)\ldd\al(x,x))\with\al(x,x)$,
\item \label{sim-def:tran} (transitivity) \ $(\al(y,z)\ldd\al(y,y))\with\al(x,y)\leq\al(x,z)$
\end{enumerate}
for all $x,y,z\in X$.
\end{defn}

Note that $\al(x,x)$ is hermitian for all $x\in X$ by \ref{sim-def:sym}. Moreover, in the presence of \ref{sim-def:sym}, the axiom \ref{sim-def:div} of divisibility implies that $\al(x,y)=\al(y,y)\with(\al(y,y)\rdd\al(x,y))$ (see \cite[Lemma 2.3]{Hoehle2011a}), and the axiom \ref{sim-def:tran} implies that $\al(y,z)\with(\al(y,y)\rdd\al(x,y))\leq\al(x,z)$.

An easy analysis of the axioms in Definition \ref{sim-def} tells us that \ref{sim-def:str} is implied by \ref{sim-def:div} if $\sQ$ is integral, and \ref{sim-def:str} is equivalent to \ref{sim-def:div} if $\sQ$ is divisible:

\begin{prop} \label{sim-divisible} (See \cite{Hoehle2011a}.)
If $\sQ$ is a divisible quantale, then the axiom of strictness is equivalent to the axiom of divisibility. Hence, a map $\al:X\times X\to\sQ$ defines a $\sQ$-valued similarity on a set $X$ if, and only if,
\begin{enumerate}[label={\rm(\arabic*)}]
\item \label{sim-divisible:str} $\al(x,y)\leq\al(x,x)\wedge\al(y,y)$,
\item \label{sim-divisible:sym} $\al(x,y)=\al(y,x)^{\circ}$,
\item \label{sim-divisible:tran} $(\al(y,z)\ldd\al(y,y))\with\al(x,y)\leq\al(x,z)$
\end{enumerate}
for all $x,y,z\in X$.
\end{prop}

Moreover, both \ref{sim-def:str} and \ref{sim-def:div} are subsumed by \ref{sim-def:sym} and \ref{sim-def:tran} when $\sQ$ is a frame, in which case a set with a $\sQ$-valued similarity is exactly a frame-valued set in the sense of Fourman--Scott \cite{Fourman1979} and Higgs \cite{Higgs1970,Higgs1984}:

\begin{prop} \label{sim-frame} (See \cite{Fourman1979}.)
If $\sQ$ is a frame, then a map $\al:A\times A\to\sQ$ defines a $\sQ$-valued similarity on a set $A$ if, and only if,
\begin{enumerate}[label={\rm(\arabic*)}]
\item \label{sim-frame:sym} $\al(x,y)=\al(y,x)$,
\item \label{sim-frame:tran} $\al(y,z)\wedge\al(x,y)\leq\al(x,z)$
\end{enumerate}
for all $x,y,z\in A$.
\end{prop}

\begin{Discus}[of the axioms of $\sQ$-valued similarity]
Let $\al:X\times X\to\sQ$ be a $\sQ$-valued similarity.
\begin{itemize}
\item[\ref{sim-def:str}] The value
    \[\al(x,y)\]
     is understood as the truth-value of the statement that \emph{$x$ is similar to $y$}. Since each entity is supposed to be similar to itself as long as it exists (or, once it is defined), the value $\al(x,x)$ may be understood as the \emph{extent of existence} \cite{Fourman1979,Hoehle2011a} of $x$. The axiom of strictness then indicates that each entity is more similar to itself than to any other one.
\item[\ref{sim-def:sym}] Similarity is symmetric; that is, if $x$ is similar to $y$, then $y$ is similar to $x$.
\item[\ref{sim-def:div}] The value
    \[\al(x,y)\ldd\al(x,x)\]
    measures to what extent the existence of $x$ forces $x$ to be similar to $y$. The equation
    \[(\al(x,y)\ldd\al(x,x))\with\al(x,x)=\al(x,y)\]
    says that $x$ is similar to $y$ if, and only if, $x$ has been proved to exist and the existence of $x$ forces $x$ to be similar to $y$.
\item[\ref{sim-def:tran}] The inequality
    \[(\al(y,z)\ldd\al(y,y))\with\al(x,y) \leq\al(x,z)\]
    says that if $x$ is similar to $y$, and if the existence of $y$ forces $y$ to be similar to $z$, then $x$ is similar to $z$. So, this axiom refers to the transitivity of similarity.
\end{itemize}
\end{Discus}

\begin{rem} \label{Q-preorder}
A map $\al:X\times X\to\sQ$ satisfying \ref{sim-def:div} and \ref{sim-def:tran} actually defines a \emph{$\sQ$-preorder} on the \emph{$\sQ$-subset} $(X,\mu)$ with
\[\mu:X\to\sQ,\quad\mu(x)=\al(x,x);\]
see \cite{GutierrezGarcia2018,Pu2012}. Hence, $\sQ$-valued similarities are a special kind of $\sQ$-preordered $\sQ$-subsets. In particular, if
\[\al(x,x)=k\]
for all $x\in X$, then $(X,\al)$ reduces to a $\sQ$-preorder on the (crisp) set $X$; see, e.g., \cite{Bvelohlavek2004,Hoehle2015,Lai2006}.
\end{rem}

\begin{rem} \label{indistinguishability}
If $\sQ$ is an integral quantale, let $\al:X\times X\to\sQ$ be a map with
\[\al(x,x)=k=\top\]
for all $x\in X$. Then $\al$ is a $\sQ$-valued similarity on $X$ if, and only if,
\begin{enumerate}[label=(\arabic*)]
\item (symmetry) \ $\al(x,y)=\al(y,x)^{\circ}$,
\item (transitivity) \ $\al(y,z)\with\al(x,y)\leq\al(x,z)$
\end{enumerate}
for all $x,y,z\in X$. Hence, for an integral quantale $\sQ$, $\sQ$-valued similarities $\al$ with $\al(x,x)=k=\top$ for all $x\in X$ generalize \emph{probabilistic relations} in the sense of Menger \cite{Menger1951}, \emph{similarity relations} in the sense of Zadeh \cite{Zadeh1971}, \emph{likeness relations} in the sense of Ruspini \cite{Ruspini1982} and \emph{indistinguishability operators} in the sense of Trillas--Valverde \cite{Trillas1984,Valverde1985}.
\end{rem}

\begin{exmp} \label{sim-2}
For $\sQ={\bf 2}$, a ${\bf 2}$-valued similarity $\al$ on a set $X$ is just an equivalence relation on a subset of $X$. Explicitly,
\[\{(x,y)\in X\times X\mid \al(x,y)=1\}\]
is an equivalence relation on the subset $\{x\mid\al(x,x)=1\}$ consisting of elements that ``have been defined''.
\end{exmp}

\begin{exmp}[Guiding example] \label{sim-exmp-PCX}
Let $X$ be a topological space. Then it follows from Proposition \ref{sim-frame} that the map
\[\al: \PCX\times \PCX\to\OX\]
given by Equation \eqref{sim-PCX} is an $\OX$-valued similarity on the set $\PCX$ of partially defined real-valued continuous maps on $X$ (see Equation \eqref{PCX-def}).
\end{exmp}

\begin{exmp} \label{sim-exmp-real}
Every \emph{partial metric space} (cf. \cite[Definition 3.1]{Matthews1994} and \cite[Definition 2]{Bukatin2009}) is a $[0,\infty]$-valued similarity. A \emph{generalized} partial metric space (cf. \cite[Example 3.10]{Pu2012} and \cite[Example 2.14]{Stubbe2014}) becomes a $[0,\infty]$-valued similarity whenever it is symmetric. Explicitly, a (generalized) partial metric space is a set $X$ equipped with a map
\[\al:X\times X\to[0,\infty]\]
such that
\[\al(x,x)\vee \al(y,y)\leq \al(x,y)\quad\text{and}\quad\al(x,z)\leq\al(y,z)-\al(y,y)+\al(x,y)\]
for all $x,y,z\in X$. As a concrete instance of such examples, let
\[\CI=\{[a,b]\mid 0\leq a<b\leq\infty\}\]
be the set of closed intervals contained in $[0,\infty]$.  Then
\[\al([a,b],[c,d])=b\vee d-a\wedge c\]
defines a $[0,\infty]$-valued similarity on $\CI$; that is, $(\CI,\al)$ is a symmetric (generalized) partial metric space.
\end{exmp}

Now we are ready to present the key notion of this paper:

\begin{defn} \label{dissim-def}
A \emph{$\sQ$-valued dissimilarity} on a set $X$ is a map
\[\be:X\times X\to\sQ\]
such that
\begin{enumerate}[label=(D\arabic*)]
\item \label{dissim-def:str} (strictness) \ $\be(x,y)\geq\be(x,x)\vee\be(y,y)$,
\item \label{dissim-def:sym} (symmetry) \ $\be(x,y)=\be(y,x)^{\circ}$,
\item \label{dissim-def:reg} (regularity) \ $\be(x,y)=\be(x,x)\ldd(\be(x,y)\rdd \be(x,x))$,
\item \label{dissim-def:tran} (contrapositive transitivity) \ $\be(x,z)\leq\be(x,y)\ldd(\be(y,z)\rdd\be(y,y))$
\end{enumerate}
for all $x,y,z\in X$.
\end{defn}

Note that $\be(x,x)$ is hermitian for all $x\in X$ by \ref{dissim-def:sym}. Moreover, in the presence of \ref{dissim-def:sym}, the axiom \ref{dissim-def:reg} implies that $\be(x,y)=(\be(y,y)\ldd\be(x,y))\rdd \be(y,y)$, and the axiom \ref{dissim-def:tran} implies that $\be(x,z)\leq(\be(y,y)\ldd\be(x,y))\rdd\be(y,z)$. With a direct computation it is easy to see that \ref{dissim-def:str} is implied by \ref{dissim-def:reg} if $\sQ$ is integral, and \ref{dissim-def:str} is equivalent to \ref{dissim-def:reg} if $\sQ$ is a complete MV-algebra.

\begin{Discus}[of the axioms of $\sQ$-valued dissimilarity]
Let $\be:X\times X\to\sQ$ be a $\sQ$-valued dissimilarity.
\begin{itemize}
\item[\ref{dissim-def:str}] The value
    \[\be(x,y)\]
     is understood as the truth-value of the statement that \emph{$x$ is dissimilar to $y$}. The axiom of strictness dictates that each entity is less dissimilar to itself than to any other one, which is parallel to the assertion that each entity is more similar to itself than to any other one.

     Since each entity is supposed to be similar to itself unless it is still \emph{undefined}, the value $\be(x,x)$ may be understood as the \emph{extent of $x$ being undefined}, or the \emph{degree of non-existence} of $x$. Therefore, the underlying logical principle of the axiom of strictness is that \emph{non-existence implies dissimilarity}.
\item[\ref{dissim-def:sym}] Dissimilarity is symmetric; that is, if $x$ is dissimilar to $y$, then $y$ is dissimilar to $x$.
\item[\ref{dissim-def:reg}] The value
    \[\be(x,y)\rdd \be(x,x)\]
    measures the extent that the dissimilarity between $x$ and $y$ forces $x$ to be undefined; in other words, it is the truth-value of the \emph{contrapositive} of the assertion that ``if $x$ is defined, then $x$ similar to $y$". The equation
    \begin{equation} \label{regularity}
    \be(x,y)=\be(x,x)\ldd(\be(x,y)\rdd\be(x,x))
    \end{equation}
    then asserts that $x$ is dissimilar to $y$ if, and only if, $x$ being ``similar'' to $y$ would force $x$ to be undefined.

    In order to explain the name ``regularity'' of this axiom, let us recall that in a frame $\sQ$, an element $q\in\sQ$ is \emph{regular} \cite{Johnstone1986} if
    \begin{equation} \label{regular-def}
    \bot\ldd(q\rdd\bot)=q=(\bot\ldd q)\rdd\bot.
    \end{equation}
    The term ``regular'' stems from the fact that regular open sets in a topological space $X$ are exactly regular elements in the frame $\OX$. Analogously, in a quantale $\sQ$ we may call an element $q\in\sQ$ \emph{regular} if
    \[\bot\ldd(q\rdd\bot)=q=(\bot\ldd q)\rdd\bot.\]
    If $\sQ$ is integral and $r\in\sQ$, then it is easy to verify that the operation
    \[p\with_r q:=(p\with q)\vee r\]
    defines a quantale structure on $\ua\!r:=\{q\in\sQ\mid r\leq q\}$, and regular elements in this quantale are precisely those $q\in\ua\!r$ satisfying
    \[r\ldd(q\rdd r)=q=(r\ldd q)\rdd r.\]
    Hence, with a slight abuse of language, it makes sense to read \eqref{regularity} as ``$\be(x,y)$ is \emph{regular} with respect to $\be(x,x)$''.
\item[\ref{dissim-def:tran}] The inequality
    \[\be(x,z)\leq\be(x,y)\ldd(\be(y,z)\rdd\be(y,y))\]
    is equivalent to
    \[\be(x,z)\with(\be(y,z)\rdd\be(y,y))\leq\be(x,y),\]
    which claims that if $x$ is dissimilar to $z$, and if the dissimilarity between $y$ and $z$ forces $y$ to be undefined, then $x$ is dissimilar to $y$; in other words, if $x$ dissimilar to $z$ and $y$ is ``similar'' to $z$, then $x$ is dissimilar to $y$. So, this axiom is actually the \emph{contrapositive transitivity} of dissimilarity.
\end{itemize}
\end{Discus}

\begin{exmp} \label{dissim-2}
For $\sQ={\bf 2}$, a ${\bf 2}$-valued dissimilarity $\be$ on a set $X$ is the complement of an equivalence relation on a subset of $X$. Explicitly,
\[\{(x,y)\in X\times X\mid \be(x,y)=1\}\]
is the complement (in $X\times X$) of an equivalence relation on the subset $\{x\mid\be(x,x)=0\}$ consisting of elements that ``have been defined''. So, as one expects, in this case each dissimilarity relation is the negation of a similarity relation, and vice versa.
\end{exmp}

\begin{exmp}[Guiding example] \label{dissim-exmp-PCX}
Let $\PCX$ be given as in Example \ref{sim-exmp-PCX}, and define
\[\be(f,g):=\Int(X-\Int\{x\in D(f)\cap D(g)\mid f(x)=g(x)\})\]
for all $f,g\in\PCX$. Then, one can check, via a straightforward but quite lengthy verification, that $\be$ is an $\OX$-valued dissimilarity on $\PCX$. The conclusion is also an immediate consequence of Example \ref{sim-exmp-PCX} and Proposition \ref{frame-neg-sim-dissim} that will be explained later. It is clear that $\be(f,f)$ is the largest open set on which $f$ is \emph{undefined}. 
\end{exmp}

\begin{exmp} \label{dissim-exmp-real}
Let
\[\CI=\{[a,b]\mid 0\leq a<b\leq\infty\}\]
as in Example \ref{sim-exmp-real}, and define
\[\be([a,b],[c,d])=\begin{cases}
0 & \text{if}\ b\vee d=\infty,\\
\max\{0,b\wedge d-a\vee c\} & \text{else}.
\end{cases}\]
Then it is straightforward to check that $\be$ is a $[0,\infty]$-valued dissimilarity on $\CI$.
\end{exmp}

\begin{rem} \label{dissim-bot} (to be continued in Remark \ref{dissim-bot-2})
Let $\be:X\times X\to\sQ$ be a map with
\[\be(x,x)=\bot\]
for all $x\in X$. Then $\be$ is a $\sQ$-valued dissimilarity on $X$ if, and only if,
\begin{enumerate}[label=(\arabic*)]
\item (symmetry) \ $\be(x,y)=\be(y,x)^{\circ}$,
\item (regularity) \ $\be(x,y)$ is a regular element of $\sQ$ (see Equation \eqref{regular-def}),
\item (contrapositive transitivity) \ $\be(x,z)\leq\be(x,y)\ldd(\be(y,z)\rdd\bot)$
\end{enumerate}
for all $x,y,z\in X$. The semantic meaning of the inequality $\be(x,z)\leq\be(x,y)\ldd(\be(y,z)\rdd\bot)$ is that ``if $x$ is dissimilar to $z$ and if $y$ is not dissimilar to $z$, then $x$ is dissimilar to $y$.'' In particular, if $\sQ$ is a complete Boolean algebra, then, in the presence of the axiom of symmetry, the axiom of contrapositive transitivity is actually equivalent to \[\be(x,z)\leq\be(x,y)\vee\be(y,z),\]
which means that ``if $x$ is dissimilar to $z$, then for each $y$, either $x$ is dissimilar to $y$ or $y$ is dissimilar to $z$.''

In what follows, a $\sQ$-valued dissimilarity $\beta$ with
\[\beta(x,x)=\bot\]
for all $x$ will be called \emph{rigid}. In a rigid $\sQ$-valued dissimilarity, every entity is never dissimilar to itself no matter whether it has been ``fully defined''.
\end{rem}

\begin{rem} \label{dissim-vs-apart}
Although our notion of dissimilarity is inspired by that of \emph{apartness relation} of Scott (see \cite[Section 4]{Scott1979}), they are conceptually different. If $\sQ$ is a frame, then a $\sQ$-valued model of apartness relation consists of the following data:
\begin{itemize}
\item a set $X$;
\item a map $E:X\to\sQ$, where the value $E(x)$ is interpreted as the \emph{extent of existence} of $x$;
\item a map $\ga:X\times X\to\sQ$, where the value $\ga(x,y)$ is interpreted as the degree of $x$ being \emph{apart} from $y$.
\end{itemize}
These data are subject to the following requirements for all $x,y,z\in X$:
\begin{enumerate}[label=(\arabic*)]
\item  $\ga(x,y)\leq E(x)\wedge E(y)$,
\item  $\ga(x,x)=\bot$,
\item  $\ga(x,y)=\ga(y,x)$,
\item  $\ga(x,z)\wedge E(y)\leq \ga(x,y)\vee\ga(z,y)$.
\end{enumerate}
It is easy to see that the map $\be$ given in Example \ref{dissim-exmp-PCX} cannot be made into an $\OX$-valued apartness relation on $\PCX$, and thus $\sQ$-valued apartness relations are essentially different from $\sQ$-valued similarities. However, they are closely related if $\sQ$ is a complete Boolean algebra as we see below.

Let ${\sf B}$ be a complete Boolean algebra. Then a rigid ${\sf B}$-valued dissimilarity on a set $X$ is a map $\be:X\times X\to{\sf B}$ such that
\[\be(x,x)=\bot\quad\text{and}\quad\be(x,z)\leq\be(x,y)\vee\be(y,z)\]
for all $x,y,z\in X$. If $(X,E,\ga)$ is a $\sB$-valued apartness relation with $E(x)=\top$ for all $x\in X$, then $\ga$ is a rigid $\sB$-valued similarity on $X$. 
Conversely, if $\be$ is a rigid $\sB$-valued dissimilarity on $X$, then $(X,E,\be)$ is a $\sB$-valued apartness relation with $E(x)=\top$ for all $x\in X$.
Therefore, for a complete Boolean algebra $\sB$, a $\sB$-valued apartness relation on a set whose elements have all been ``proved to exist'' is precisely a rigid $\sB$-valued dissimilarity relation.

For connections between Boolean-valued similarities Boolean-valued apartness relations, see Remark \ref{sim-vs-apart-Boolean}.
\end{rem}

\section{Similarities and dissimilarities as enriched categories} \label{Sim-Dissim-Cat}

It is already known from \cite{Hoehle2011a} that sets equipped with a $\sQ$-valued similarity are \emph{symmetric} categories enriched in a subquantaloid of the quantaloid $\DsQ$ of \emph{diagonals} of $\sQ$ \cite{Hoehle2011a,Pu2012,Stubbe2014}. The aim of this section is to show that there is an analogous categorical interpretation for $\sQ$-valued dissimilarities; that is, a set equipped with a $\sQ$-dissimilarity can be made into a \emph{symmetric} category enriched in a subquantaloid $\KsQ$ of the quantaloid $\BsQ$ of \emph{back diagonals} of $\sQ$ introduced in \cite{Shen2016a}. Therefore, $\sQ$-valued similarities and $\sQ$-valued dissimilarities are both instances of enriched categories.

\subsection{Quantaloid-enriched categories}

A \emph{quantaloid} \cite{Rosenthal1996} $\CQ$ is a category in which every hom-set is a complete lattice, and the composition $\circ$ of $\CQ$-arrows preserves suprema on both sides, i.e.,
\[v\circ\Big(\bv_{i\in I} u_i\Big)=\bv_{i\in I}v\circ u_i\quad\text{and}\quad\Big(\bv_{i\in I} v_i\Big)\circ u=\bv_{i\in I}v_i\circ u\]
for all $\CQ$-arrows $u,u_i:p\to q$, $v,v_i:q\to r$ $(i\in I)$. The corresponding right adjoints induced by the composition maps \[-\circ u\dv -\lda u:\ \CQ(p,r)\to\CQ(q,r)\quad\text{and}\quad v\circ -\dv v\rda -:\ \CQ(p,r)\to\CQ(p,q)\] satisfy \[v\circ u\leq w\iff v\leq w\lda u\iff u\leq v\rda w\] for all $\CQ$-arrows $u:p\to q$, $v:q\to r$, $w:p\to r$, where the operations $\lda$ and $\rda$ are called \emph{left} and \emph{right implications} in $\CQ$, respectively.

A (unital) quantale $\sQ=(\sQ,\with,k)$ is exactly a one-object quantaloid. As we will construct several quantaloids out of a quantale $\sQ$ later, in order to eliminate ambiguity we denote implications in a quantale $\sQ$ by $\ldd$ and $\rdd$ as in \eqref{imp-def}, and reserve the notations $\lda$ and $\rda$ for the quantaloids constructed from $\sQ$.

Given a \emph{small} quantaloid $\CQ$ (i.e., $\ob\CQ$ is a set), a \emph{$\CQ$-category} (also \emph{category enriched in $\CQ$}) \cite{Stubbe2005} consists of a \emph{$\CQ$-typed set} $X$ (i.e., a set  $X$ equipped with a \emph{type} map $|\text{-}|:X\to\ob\CQ$) and a family of $\CQ$-arrows $\al(x,y)\in\CQ(|x|,|y|)$ $(x,y\in X)$, such that
\[1_{|x|}\leq\al(x,x)\quad\text{and}\quad\al(y,z)\circ\al(x,y)\leq\al(x,z)\]
for all $x,y,z\in X$.

A \emph{$\CQ$-functor} $f:(X,\al)\to(Y,\be)$ between $\CQ$-categories is a map $f:X\to Y$ such that
\[|x|=|fx|\quad\text{and}\quad\al(x,y)\leq\be(fx,fy)\]
for all $x,y\in X$. The category of $\CQ$-categories and $\CQ$-functors is denoted by
\[\QCat.\]

A \emph{homomorphism} $F:\CQ\to\CR$ of quantaloids is a functor of the underlying categories that preserves suprema of $\CQ$-arrows. By an \emph{involution} on a quantaloid $\CQ$ we mean a homomorphism
\[(-)^{\circ}:\CQ^{\op}\to\CQ\]
of quantaloids whose composition with itself outputs the identity homomorphism on $\CQ$. Explicitly, an involution on $\CQ$ is given by maps
\[(-)^{\circ}:\ob\CQ\to\ob\CQ\quad\text{and}\quad (-)^{\circ}:\CQ(p,q)\to\CQ(q^{\circ},p^{\circ})\]
for all $p,q\in\ob\CQ$, such that
\[q^{\circ\circ}=q,\quad (1_q)^{\circ}=1_{q^{\circ}},\quad u^{\circ\circ}=u,\quad (v\circ u)^{\circ}=u^{\circ}\circ v^{\circ}\quad\text{and}\quad\Big(\bv_{i\in I}u_i\Big)^{\circ}=\bv_{i\in I}u_i^{\circ}\]
for all $q\in\ob\CQ$ and $\CQ$-arrows $u,u_i:p\to q$ $(i\in I)$. Given a small \emph{involutive} quantaloid $\CQ$, i.e., a small quantaloid $\CQ$ equipped with an involution, we say that a $\CQ$-category $(X,\al)$ is \emph{symmetric} if
\begin{equation} \label{sym-Q-cat}
\al(x,y)=\al(y,x)^{\circ}
\end{equation}
for all $x,y\in X$.

\begin{rem} \label{sym-Q-cat-Stubbe}
Our definition of involutive quantaloids here slightly generalizes that of Rosenthal (see \cite[Definition 2.5.1]{Rosenthal1996}), which requires an involution to be the identity on objects. Indeed, as explained below, these definitions make no difference for the purpose of defining the symmetry of $\CQ$-categories.

Let $(X,\al)$ be a symmetric $\CQ$-category. Since $\al(x,y)\in\CQ(|x|,|y|)$ and $\al(y,x)^{\circ}\in\CQ(|x|^{\circ},|y|^{\circ})$, Equation \eqref{sym-Q-cat} actually forces
\[|x|=|x|^{\circ}\]
for all $x\in X$. Therefore, a symmetric $\CQ$-category is in fact a category enriched in the full subquantaloid $\CQ^{\circ}$ of $\CQ$ with
\[\ob\CQ^{\circ}=\{q\in\ob\CQ\mid q=q^{\circ}\},\]
which is equipped with the involution inherited from $\CQ$ that is clearly neutral on objects. Hence, symmetric $\CQ$-categories defined by \eqref{sym-Q-cat} are precisely symmetric $\CQ^{\circ}$-categories as postulated by \cite[Definition 2.3]{Heymans2011} and \cite[Definition 6.2]{Hoehle2011a}, whose prototype comes from \cite{Betti1982a}.
\end{rem}

\subsection{$\sQ$-valued similarities as enriched categories}

In this subsection we recall how $\sQ$-valued similarities are represented as symmetric categories enriched in a quantaloid $\HsQ$ (see \cite{Hoehle2011a}), which is a subquantaloid of the quantaloid $\DsQ$ of \emph{diagonals} of $\sQ$ \cite{Hoehle2011a,Stubbe2014}.

Let $p,q\in\sQ$. By a \emph{diagonal} \cite{Stubbe2014} from $p$ to $q$ we mean an element $d\in\sQ$ such that
\begin{equation} \label{diagonal-def}
(d\ldd p)\with p=d=q\with(q\rdd d).
\end{equation}

\begin{lem} \label{diagonal-property}
Let $p,q\in\sQ$.
\begin{enumerate}[label=\rm(\arabic*)]
\item \label{diagonal-property:bot} $\bot$ is a diagonal from $p$ to $q$.
\item \label{diagonal-property:id} $q$ is a diagonal from $q$ to $q$.
\item \label{diagonal-property:inf} $\bv\limits_{i\in I}d_i$ is a diagonal from $p$ to $q$ if so is each $d_i$ $(i\in I)$.
\item \label{diagonal-property:involution} If $d$ is a diagonal from $p$ to $q$, then $d^{\circ}$ is a diagonal from $q^{\circ}$ to $p^{\circ}$.
\end{enumerate}
\end{lem}

If $d$ is a diagonal from $p$ to $q$ and $e$ is a diagonal from $q$ to $r$, then it is not difficult to verify that
\[(e\ldd q)\with d=e\with(q\rdd d),\]
and thus we set
\begin{equation} \label{diagonal-comp-def}
e\diamond d:=(e\ldd q)\with d=e\with(q\rdd d).
\end{equation}

\begin{lem} \label{diagonal-comp}
Let $d$ be a diagonal from $p$ to $q$ and let $e$ be a diagonal from $q$ to $r$.
\begin{enumerate}[label=\rm(\arabic*)]
\item $e\diamond d$ is a diagonal from $p$ to $r$, called the \emph{composite} of $d$ and $e$.
\item $d\diamond p=d=q\diamond d$.
\item The composition of diagonals is associative. 
\item The composition of diagonals preserves suprema on both sides, i.e.,
    \[e\diamond\bv\limits_{i\in I}d_i=\bv_{i\in I}(e\diamond d_i)\quad\Big(\bv\limits_{i\in I} e_i\Big)\diamond d=\bv_{i\in I}(e_i\diamond d)\]
    for all diagonals $d_i$ from $p$ to $q$ and $e_i$ from $q$ to $r$ $(i\in I)$.
\end{enumerate}
\end{lem}

Lemmas \ref{diagonal-property} and \ref{diagonal-comp} guarantee the existence of a quantaloid $\DsQ$ given by the following data, called the quantaloid of \emph{diagonals} of $\sQ$:
\begin{itemize}
\item objects of $\DsQ$ are elements $p,q,r,\dots$ of $\sQ$;
\item for $p,q\in\sQ$, morphisms from $p$ to $q$ in $\DsQ$ are diagonals from $p$ to $q$;
\item the composition of diagonals $d\in\BsQ(p,q)$ and $e\in\BsQ(q,r)$ is given by $e\diamond d$;
\item the identity diagonal on $q\in\sQ$ is $q$ itself;
\item each hom-set $\DsQ(p,q)$ is equipped with the order inherited from $\sQ$.
\end{itemize}
As pointed out in \cite[Example 2.14]{Stubbe2014}, the above construction makes sense not only for a general quantaloid $\CQ$, but also for a general category $\CC$ \cite{Grandis2000,Grandis2002}. It is easily seen that
\[\HsQ(p,q):=\{d\in\DsQ(p,q)\mid d\leq p\wedge q\}\]
for all $p,q\in\sQ$ defines a subquantaloid $\HsQ$ of $\DsQ$ (see \cite[Remark 4.4]{Hoehle2011a}), and we denote by
\[d:p\rqa q\]
a morphism $d\in\HsQ(p,q)$; that is,
\[d:p\rqa q\iff d\leq p\wedge q\quad\text{and}\quad(d\ldd p)\with p=d=q\with(q\rdd d).\]
In the case that $\sQ$ is integral, we have
\[\HsQ=\DsQ\]
since $d\leq p\wedge q$ would be a consequence of $(d\ldd p)\with p=d=q\with(q\rdd d)$. Moreover,
\[\HsQ(p,q)=\DsQ(p,q)=\{d\in\sQ\mid d\leq p\wedge q\}\]
if $\sQ$ is divisible.

\begin{exmp}
Since Lawvere's quantale $[0,\infty]$ (see Example \ref{quantale-exmp}\ref{quantale-exmp:Lawvere}) is divisible, it holds that
\[\BD_*[0,\infty](p,q)=[p\vee q,\infty]\]
for all $p,q\in[0,\infty]$, and
\[e\diamond d=e-q+d \]
for all $d:p\rqa q$ and $e:q\rqa r$.
\end{exmp}

Note that by Lemma \ref{diagonal-property}\ref{diagonal-property:involution}, $\HsQ$ is also an involutive quantaloid with the involution lifted from $\sQ$. From the definition we see that a $\HsQ$-category consists of a set $X$, a map $|\text{-}|:X\to\sQ$ and a map $\al:X\times X\to\sQ$ such that
\begin{enumerate}[label=(\arabic*)]
\item \label{HsQ-cat:str} $\al(x,y)\leq|x|\wedge|y|$,
\item \label{HsQ-cat:div} $(\al(x,y)\ldd |x|)\with |x|=\al(x,y)=|y|\with(|y|\rdd\al(x,y))$,
\item \label{HsQ-cat:ref} $|x|\leq\al(x,x)$,
\item \label{HsQ-cat:tran} $(\al(y,z)\ldd |y|)\with\al(x,y)=\al(y,z)\with(|y|\rdd\al(x,y))\leq\al(x,z)$
\end{enumerate}
for all $x,y,z\in X$, where \ref{HsQ-cat:str} and \ref{HsQ-cat:div} follows from $\al(x,y)\in\HsQ(|x|,|y|)$. Then, \ref{KsQ-cat:str} in conjunction with \ref{KsQ-cat:ref} leads to
\[\al(x,x)=|x|\]
for all $x\in X$, and thus a $\HsQ$-category is exactly given by a map $\al:X\times X\to\sQ$ such that (cf. Definition \ref{sim-def})
\begin{itemize}
\item $\al(x,y)\leq\al(x,x)\wedge\al(y,y)$,
\item $(\al(x,y)\ldd\al(x,x))\with\al(x,x)=\al(x,y)=\al(y,y)\with(\al(y,y)\rdd\al(x,y))$,
\item $(\al(y,z)\ldd\al(y,y))\with\al(x,y)=\al(y,z)\with(\al(y,y)\rdd\al(x,y))\leq\al(x,z)$
\end{itemize}
for all $x,y,z\in X$. Therefore, a $\sQ$-valued similarity $\al$ on a set $X$ is exactly a $\HsQ$-category satisfying
\begin{itemize}
\item $\al(x,y)=\al(y,x)^{\circ}$
\end{itemize}
for all $x,y\in X$; that is, a symmetric $\HsQ$-category:

\begin{thm} \label{sim-DsQ-Cat} (See \cite{Hoehle2011a}.)
A set equipped with a $\sQ$-valued similarity is precisely a symmetric $\HsQ$-category.
\end{thm}

\begin{rem} \label{sym-DQ-cat}
As elaborated in Remark \ref{sym-Q-cat-Stubbe}, a symmetric $\HsQ$-category is exactly a symmetric $\HssQ$-category, where $\HssQ$ is the involutive quantaloid constructed from $\sQ$ by H{\"o}hle--Kubiak in \cite[Proposition 6.3]{Hoehle2011a}.
\end{rem}

\subsection{$\sQ$-valued dissimilarities as enriched categories}

In this subsection, we construct a quantaloid $\KsQ$ for each quantale $\sQ$ and reveal that sets equipped with a $\sQ$-valued dissimilarity are precisely \emph{symmetric} categories enriched in $\KsQ$.

Let $p,q\in\sQ$. By a \emph{back diagonal} \cite{Shen2016a} from $p$ to $q$ we mean an element $b\in\sQ$ such that
\begin{equation} \label{back-diagonal-def}
p\ldd(b\rdd p)=b=(q\ldd b)\rdd q.
\end{equation}
The verification of the following lemma is straightforward:

\begin{lem} \label{back-diagonal-property}
Let $p,q\in\sQ$.
\begin{enumerate}[label=\rm(\arabic*)]
\item \label{back-diagonal-property:top} $\top$ is a back diagonal from $p$ to $q$.
\item \label{back-diagonal-property:id} $q$ is a back diagonal from $q$ to $q$.
\item \label{back-diagonal-property:inf} $\bw\limits_{i\in I}b_i$ is a back diagonal from $p$ to $q$ if so is each $b_i$ $(i\in I)$.
\item \label{back-diagonal-property:involution} If $b$ is a back diagonal from $p$ to $q$, then $b^{\circ}$ is a back diagonal from $q^{\circ}$ to $p^{\circ}$.
\end{enumerate}
\end{lem}

If $b$ is a back diagonal from $p$ to $q$ and $c$ is a back diagonal from $q$ to $r$, then
\[b\ldd(c\rdd q)\leq(q\ldd b)\rdd c\]
since
\[(q\ldd b)\with(b\ldd(c\rdd q))\leq q\ldd(c\rdd q)=c,\]
and similarly
\[(q\ldd b)\rdd c\leq b\ldd(c\rdd q).\]
Thus it makes sense to define
\begin{equation} \label{back-diagonal-comp-def}
c\bullet b:=b\ldd(c\rdd q)=(q\ldd b)\rdd c,
\end{equation}
which turns out to be a back diagonal from $p$ to $r$:

\begin{lem} \label{back-diagonal-comp}
Let $b$ be a back diagonal from $p$ to $q$ and let $c$ be a back diagonal from $q$ to $r$.
\begin{enumerate}[label=\rm(\arabic*)]
\item $c\bullet b$ is a back diagonal from $p$ to $r$, called the \emph{composite} of $b$ and $c$.
\item $b\bullet p=b=q\bullet b$.
\item The composition of back diagonals is associative. 
\item The composition of back diagonals preserves infima on both sides, i.e.,
    \[c\bullet\bw_{i\in I}b_i=\bw_{i\in I}(c\bullet b_i)\quad\Big(\bw_{i\in I} c_i\Big)\bullet b=\bw_{i\in I}(c_i\bullet b)\]
    for all back diagonals $b_i$ from $p$ to $q$ and $c_i$ from $q$ to $r$ $(i\in I)$.
\end{enumerate}
\end{lem}

From Lemmas \ref{back-diagonal-property} and \ref{back-diagonal-comp} we actually obtain a quantaloid $\BsQ$ from each quantale $\sQ$, called the quantaloid of \emph{back diagonals} of $\sQ$:
\begin{itemize}
\item objects of $\BsQ$ are elements $p,q,r,\dots$ of $\sQ$;
\item for $p,q\in\sQ$, morphisms from $p$ to $q$ in $\BsQ$ are back diagonals from $p$ to $q$;
\item the composition of back diagonals $b\in\BsQ(p,q)$ and $c\in\BsQ(q,r)$ is given by $c\bullet b$;
\item the identity back diagonal on $q\in\sQ$ is $q$ itself;
\item each hom-set $\BsQ(p,q)$ is equipped with the \emph{reversed} order inherited from $\sQ$.
\end{itemize}
It should be noted that the construction of $\BsQ$ makes sense not only for a quantale $\sQ$, but also for a general quantaloid $\CQ$; see \cite{Shen2016a}.

For each $p,q\in\sQ$, let
\[\KsQ(p,q):=\{b\in\BsQ\mid p\vee q\leq b\};\]
that is, $\KsQ(p,q)$ consists of back diagonals from $p$ to $q$ that are above both $p$ and $q$. Then, it is easy to see that $\KsQ$ is a subquantaloid of $\BsQ$, and we write
\[b:p\lar q\]
for a morphism $b\in\KsQ(p,q)$; that is,
\[b:p\lar q\iff p\vee q\leq b\quad\text{and}\quad p\ldd(b\rdd p)=b=(q\ldd b)\rdd q.\]
Note that if $\sQ$ is integral, then $p\vee q\leq b$ is implied by $p\ldd(b\rdd p)=b=(q\ldd b)\rdd q$, and thus
\[\KsQ=\BsQ\]
in this case.

\begin{exmp}
For Lawvere's quantale $[0,\infty]$ given in Example \ref{quantale-exmp}\ref{quantale-exmp:Lawvere},
\[\BB_*[0,\infty](p,q)=\begin{cases}
[0,p\wedge q] & \text{if}\ p,q<\infty,\\
\{0,\infty\} & \text{if}\ p=q=\infty,\\
\{0\}  & \text{else}
\end{cases}\]
for all $p,q\in[0,\infty]$, and
\[  c\bullet b=\begin{cases}
0 & \text{if}\ c\wedge b<\infty\ \text{and}\ q=\infty,\\
\max\{0,c-q+b\}  & \text{else}.
\end{cases}\]
for all $b:p\lar q$ and $c:q\lar r$.
\end{exmp}

Now let us look at categories enriched in the quantaloid $\KsQ$. As a direct consequence of Lemma \ref{back-diagonal-property}\ref{back-diagonal-property:involution}, $\KsQ$ is an involutive quantaloid with the involution lifted from $\sQ$; that is, the involution $(-)^{\circ}$ on $\sQ$ actually gives rise to an involution on $\KsQ$. By definition, a $\KsQ$-category consists of a set $X$, a map $|\text{-}|:X\to\sQ$ and a map $\be:X\times X\to\sQ$ such that
\begin{enumerate}[label=(\arabic*)]
\item \label{KsQ-cat:str} $|x|\vee |y|\leq\be(x,y)$,
\item \label{KsQ-cat:reg} $|x|\ldd(\be(x,y)\rdd|x|)=\be(x,y)=(|y|\ldd\be(x,y))\rdd |y|$,
\item \label{KsQ-cat:ref} $\be(x,x)\leq|x|$,
\item \label{KsQ-cat:tran} $\be(x,z)\leq\be(x,y)\ldd(\be(y,z)\rdd|y|)=(|y|\ldd\be(x,y))\rdd\be(y,z)$
\end{enumerate}
for all $x,y,z\in X$, where \ref{KsQ-cat:str} and \ref{KsQ-cat:reg} follows from $\be(x,y)\in\KsQ(|x|,|y|)$. Note that the combination of \ref{KsQ-cat:str} and \ref{KsQ-cat:ref} forces
\begin{equation} \label{bexx=x}
\be(x,x)=|x|
\end{equation}
for all $x\in X$, and thus a $\KsQ$-category is precisely given by a map $\be:X\times X\to\sQ$ such that (cf. Definition \ref{dissim-def})
\begin{itemize}
\item $\be(x,y)\geq\be(x,x)\vee\be(y,y)$,
\item $\be(x,x)\ldd(\be(x,y)\rdd \be(x,x))=\be(x,y)=(\be(y,y)\ldd\be(x,y))\rdd\be(y,y)$,
\item $\be(x,z)\leq\be(x,y)\ldd(\be(y,z)\rdd\be(y,y))=(\be(y,y)\ldd\be(x,y))\rdd\be(y,z)$
\end{itemize}
for all $x,y,z\in X$. Therefore, a $\sQ$-valued dissimilarity $\be$ on a set $X$ is exactly a $\KsQ$-category satisfying
\begin{itemize}
\item $\be(x,y)=\be(y,x)^{\circ}$
\end{itemize}
for all $x,y\in X$; that is, a symmetric $\KsQ$-category:

\begin{thm} \label{dissim-BsQ-Cat}
A set equipped with a $\sQ$-valued dissimilarity is precisely a symmetric $\KsQ$-category.
\end{thm}

\begin{rem} \label{dissim-bot-2}
It is clear that $\KsQ(\bot,\bot)=\BsQ(\bot,\bot)$ for every quantale $\sQ$ and an element of $\KsQ(\bot,\bot)$ is exactly a regular element of $\sQ$ (see Equation \eqref{regular-def}). Furthermore, the quantale $\KsQ(\bot,\bot)$ is integral and a rigid $\sQ$-valued dissimilarity (see Remark \ref{dissim-bot}) on a set $X$ is precisely a symmetric category structure enriched in the involutive quantale $\KsQ(\bot,\bot)$.
\end{rem}

\section{Similarity vs. dissimilarity} \label{Sim-vs-Dissim}

In classical logic, the negation of a similarity relation is a dissimilarity relation, and vice versa. (cf. Examples \ref{sim-2} and \ref{dissim-2}). It is natural to ask whether it still holds in the quantale-valued setting; that is, whether the negation of a $\sQ$-valued dissimilarity is a $\sQ$-valued similarity, and vice versa. With the help of \emph{lax functors} between the quantaloids $\HsQ$ and $\KsQ$, in this section we provide some partial answers to this question in the case that $\sQ$ is a divisible quantale, a frame or a Girard quantale.

Before proceeding on, we would like to remind the readers of the fact that although being of unequivocal importance, inquiring what is really meant by \emph{negation} remains a sensitive question in fuzzy set theory, and it will not be discussed here. In what follows we just focus on two kinds of negations in a quantale, one of which is determined by the bottom element of the quantale, and the other is the \emph{linear negation} in a Girard quantale. Both of the negations under concern are of residuation-type; that is, they are determined by the operator $\with$ via adjoint property.

Recall that a \emph{lax functor} \cite{Hofmann2014,Street1972a} $F:\CQ\to\CR$ of quantaloids is given by maps
\[F:\ob\CQ\to\ob\CR\quad\text{and}\quad F_{p,q}:\CQ(p,q)\to\CR(Fp,Fq)\]
for all $p,q\in\ob\CQ$ (with $F_{p,q}$ usually written as $F$ for short), such that
\begin{enumerate}[label=(\arabic*)]
\item \label{lax-functor:mono} $F_{p,q}$ is monotone,
\item \label{lax-functor:mor} $Fv\circ Fu\leq F(v\circ u)$,
\item \label{lax-functor:unit} $1_{Fq}\leq F1_q$
\end{enumerate}
for all $p,q,r\in\ob\CQ$ and $\CQ$-arrows $u:p\to q$, $v:q\to r$. A lax functor $F:\CQ\to\CR$ becomes a homomorphism of quantaloids if it preserves suprema of $\CQ$-arrows and the inequalities ``$\leq$'' in \ref{lax-functor:mor} and \ref{lax-functor:unit} are replaced by ``$=$'', and a homomorphism of quantaloids becomes an \emph{isomorphism} of quantaloids if so is the underlying functor.

Every lax functor $F:\CQ\to\CR$ of quantaloids induces a functor
\[\QCat\to\RCat,\]
which assigns to each $\CQ$-category $(X,|\text{-}|,\al)$ an $\CR$-category $(X,F|\text{-}|,F\al)$,  and each $\CQ$-functor $f:(X,|\text{-}|,\al)\to(Y,|\text{-}|,\be)$ will be mapped to an $\CR$-functor $f:(X,F|\text{-}|,F\al)\to(Y,F|\text{-}|,F\be)$. Therefore, the existence of a lax functor
\[\HsQ\to\KsQ\]
would allow us to construct a $\KsQ$-category from a $\HsQ$-category, and vice versa. Furthermore, if a lax functor
\[F:\HsQ\to\KsQ\]
preserves the involution of $\sQ$ in the sense that
\[F(d:p\rqa q)^{\circ}=(Fd:Fp\lar Fq)^{\circ},\quad\text{i.e.,}\quad Fq^{\circ}=(Fq)^{\circ}\quad\text{and}\quad Fd^{\circ}=(Fd)^{\circ}\]
for all $p,q\in\sQ$, $d\in\HsQ(p,q)$, then each $\sQ$-valued similarity would generate a $\sQ$-valued dissimilarity, and vice versa.


\subsection{When $\sQ$ is a divisible quantale}

In each quantale $\sQ$, we may define
\begin{equation} \label{negation-def}
\neg_l q:=\bot\ldd q\quad\text{and}\quad \neg_r q:=q\rdd\bot
\end{equation}
as the \emph{left} and \emph{right negations} of $q$, respectively, which can be unified to
\begin{equation} \label{negation-commutative-def}
\neg q:=\neg_l q=\neg_r q=q\ra\bot
\end{equation}
if the bottom element $\bot$ is \emph{cyclic} in the sense that
\[\bot\ldd q=q\rdd\bot\]
for all $q\in\sQ$. It is clear that the negation operators on $\sQ$ admit pointwise extensions to maps $X\times X\to\sQ$. The main result of this subsection is:

\begin{prop} \label{divisible-neg-dissim}
If $\sQ$ is a divisible quantale with the bottom element $\bot$ being cyclic, then the negation $\neg\be$ of each $\sQ$-valued dissimilarity $\be$ is a $\sQ$-valued similarity.
\end{prop}

Proposition \ref{divisible-neg-dissim} follows immediately from Equation \eqref{invo-imp} and the following lemma:

\begin{lem} \label{neg-BQ-divisible}
If $\sQ$ is a divisible quantale, then both the assignments
\[(b:p\lar q)\mapsto(\neg_l b:\neg_l p\rqa\neg_l q)\quad\text{and}\quad(b:p\lar q)\mapsto(\neg_r b:\neg_r p\rqa\neg_r q)\]
define lax functors $\neg_l,\neg_r:\KsQ\to\HsQ$.
\end{lem}

\begin{proof}
We only verify that $\neg_l:\KsQ\to\HsQ$ is a lax functor, and the lax functoriality of $\neg_r$ can be obtained dually.

If $b\in\KsQ(p,q)$, then it follows from $b\geq p\vee q$ that
\[\neg_l b\leq\neg_l p\wedge\neg_l q.\]
Thus $\neg_l b\in\HsQ(\neg_l p,\neg_l q)$ by the divisibility of $\sQ$.

Since
$\neg_l$ is clearly monotone on hom-sets and preserves identities,  it remains to prove that
\[\neg_l c\diamond\neg_l b\leq\neg_l(c\bullet b)\]
for all morphisms $b:p\lar q$, $c:q\lar r$ in $\KsQ$. Indeed, since $q\leq c$ and $\neg_l(c\rdd q)\leq q\ldd(c\rdd q)=c$, we have
\begin{align*}
\neg_l q\with c&=\neg_l(c\with(c\rdd q))\with c&(q\leq c)\\
&=(\neg_l(c\rdd q)\ldd c)\with c\\
&=\neg_l(c\rdd q),&(\neg_l(c\rdd q)\leq c)
\end{align*}
and consequently
\begin{align}
\neg_l c\diamond\neg_l b&=(\neg_l c\ldd\neg_l q)\with\neg_l b \nonumber\\
&=\neg_l(\neg_l q\with c)\with\neg_l b \nonumber\\
&=\neg_l\neg_l(c\rdd q)\with\neg_l b \label{neg-c-diamond-neg-b}
\end{align}
by Equation \eqref{diagonal-comp-def}. It follows that
\[(\neg_l c\diamond\neg_l b)\with(b\ldd(c\rdd q))=\neg_l\neg_l(c\rdd q)\with\neg_l b\with(b\ldd(c\rdd q))\leq\bot,\]
and therefore
\[\neg_l c\diamond\neg_l b\leq\neg_l(b\ldd(c\rdd q))=\neg_l(c\bullet b)\]
by Equation \eqref{back-diagonal-comp-def}, as desired.
\end{proof}

 \begin{cor} \label{neg-BQCat-divisible}
If $\sQ$ is a divisible quantale, then both the assignments
\[(X,\be)\mapsto(X,\neg_l\be)\quad\text{and}\quad(X,\be)\mapsto(X,\neg_r\be)\]
define functors $\KsQ\text{-}\Cat\to\HsQ\text{-}\Cat$.
\end{cor}

\subsection{When $\sQ$ is a frame}

In the case that $\sQ$ is a frame, the negation of a $\sQ$-valued dissimilarity is a $\sQ$-valued similarity, and vice versa:

\begin{prop} \label{frame-neg-sim-dissim}
If $\sQ$ is a frame, then
\begin{enumerate}[label={\rm(\arabic*)}]
\item \label{frame-neg-sim-dissim:neg-dissim} the negation $\neg\be$ of each $\sQ$-valued dissimilarity $\be$ is a $\sQ$-valued similarity, and
\item \label{frame-neg-sim-dissim:neg-sim} the negation $\neg\al$ of each $\sQ$-valued similarity $\al$ is a $\sQ$-valued dissimilarity.
\end{enumerate}
\end{prop}

Since each frame is a commutative and divisible quantale, Proposition \ref{divisible-neg-dissim} guarantees the validity of Proposition \ref{frame-neg-sim-dissim}\ref{frame-neg-sim-dissim:neg-dissim}. In fact, in this case Lemma \ref{neg-BQ-divisible} can be strengthened to the following:

\begin{lem} \label{frame-neg-BQ-DQ}
If $\sQ$ is a frame, then the assignment
\[(b:p\lar q)\mapsto(\neg b:\neg p \rqa\neg q)\]
defines a quantaloid homomorphism $\neg:\KsQ\to\HsQ$.
\end{lem}

\begin{proof}
It is clear that the assignment
\[(b:p\lar q)\mapsto(\neg b:\neg p \rqa\neg q)\]
preserves identities and local suprema. With Lemma \ref{neg-BQ-divisible} in hand, it remains to show that
\[\neg(c\bullet b)\leq\neg c\diamond\neg b=\neg\neg(c\ra q)\wedge\neg b\]
by Equation \eqref{neg-c-diamond-neg-b}; that is,
\[\neg((c\ra q)\ra b)\leq\neg\neg(c\ra q)\wedge\neg b=\neg(\neg(c\ra q)\vee b).\]
This is easy since
\[\neg(c\ra q)\leq(c\ra q)\ra b\quad\text{and}\quad b\leq(c\ra q)\ra b\]
are both obvious.
\end{proof}

Moreover, Proposition \ref{frame-neg-sim-dissim}\ref{frame-neg-sim-dissim:neg-sim} is a direct consequence of Lemma \ref{frame-neg-DQ-BQ} below. Before proceeding to prove this lemma, we point out that the open set $\be(f,g)$ given by Example \ref{dissim-exmp-PCX} is precisely the negation of the open set $\al(f,g)$ given by Example \ref{sim-exmp-PCX} in the frame $\OX$, i.e.,
\[\be(f,g)=\neg\al(f,g)=\al(f,g)\ra\varnothing.\]
So, by applying Proposition \ref{frame-neg-sim-dissim} to the $\OX$-valued similarity $\al$, the $\OX$-valued dissimilarity $\be$ on $\PCX$ is soon obtained.

\begin{lem}\label{frame-neg-DQ-BQ}
If $\sQ$ is a frame, then the assignment
\[(d:p\rqa q)\mapsto(\neg d:\neg p \lar\neg q)\]
defines a quantaloid homomorphism $\neg:\HsQ\to\KsQ$.
\end{lem}

\begin{proof}
First, if $d\in\HsQ(p,q)$, then $d=p\wedge(p\ra d)=q\wedge(q\ra d)$ by Equation \eqref{diagonal-def}. It follows that
\[\neg d=\neg(p\wedge(p\ra d))=(p\ra d)\ra\neg p=(((p\ra d)\ra\neg p)\ra\neg p)\ra\neg p=(\neg d\ra\neg p)\ra\neg p,\]
and similarly $\neg d=(\neg d\ra\neg q)\ra\neg q$. Hence $\neg d\in\BsQ(\neg p,\neg q)=\KsQ(\neg p,\neg q)$ by Equation \eqref{back-diagonal-def}.

Second, since $\neg:\HsQ\to\KsQ$ obviously preserves identities and local suprema, it remains to verify that
\[\neg e\bullet\neg d=\neg(e\diamond d)\]
for all $d:p\rqa q$ and $e:q\rqa r$.

Since frames are divisible, it follows that
\[\HsQ(p,q)=\DsQ(p,q)=\{d\in\sQ\mid d\leq p\wedge q\}\]
for all $p,q\in\sQ$, and the composite of $d:p\rqa q$ and $e:q\rqa r$ is given by
\[e\diamond d=e\wedge(q\ra d)= e\wedge q\wedge(q\ra d)=e\wedge d.\]
Thus we only need to show that
\[(d\leq p\wedge q\ \text{and}\ e\leq q\wedge r)\implies\neg(e\wedge d)=(\neg d\ra\neg q)\ra\neg e\]
because,  by definition, $\neg e\bullet\neg d=(\neg d\ra\neg q)\ra\neg e$. On one hand, $d\leq\neg d\ra\neg q$ implies that
\[\neg(e\wedge d)=d\ra\neg e\geq(\neg d\ra\neg q)\ra\neg e.\]
On the other hand,
\[\neg e\wedge(\neg d\ra\neg q)\leq\neg e\quad\text{and}\quad\neg d\wedge(\neg d\ra\neg q)\leq\neg q\leq\neg e\]
implies that
\begin{equation} \label{neg-e-neg-d-leq}
\neg e\vee\neg d\leq(\neg d\ra\neg q)\ra\neg e,
\end{equation}
and consequently
\begin{align*}
\neg(e\wedge d)&=\neg\neg\neg(e\wedge d)\\
&=\neg(\neg\neg e\wedge\neg\neg d) \\
&=\neg\neg(\neg e\vee\neg d)\\
&\leq\neg\neg((\neg d\ra\neg q)\ra\neg e)&(\text{Inequality \eqref{neg-e-neg-d-leq}})\\
&=\neg\neg\neg(e\wedge(\neg d\ra\neg q))\\
&=\neg(e\wedge(\neg d\ra\neg q))\\
&=(\neg d\ra\neg q)\ra\neg e,
\end{align*}
where the second equality holds since
\[\neg\neg(p\wedge q)=\neg\neg p\wedge\neg\neg q\]
for all elements $p,q$ in a frame (see, e.g., \cite[Exercise I.1.11(ii)]{Johnstone1986}).
\end{proof}

\begin{exmp} \label{C3-neg}
The assumptions on $\sQ$ in Propositions \ref{divisible-neg-dissim} and \ref{frame-neg-sim-dissim} are not indispensable. Note that the commutative quantale $C_3$ (see Example \ref{quantale-exmp}\ref{quantale-exmp:C3}) is not integral, and thus not divisible, and it holds that
\begin{align*}
&\neg\bot=\top,\quad\neg k=\neg\top=\bot;\\
&\HC(\bot,q)=\HC(q,\bot)=\HC(k,\top)=\HC(\top,k)=\{\bot\}\quad(q\in C_3),\\
&\HC(\top,\top)=\{\bot,\top\},\quad\HC(k,k)=\{\bot,k\};\\
&\KC(\top,q)=\KC(q,\top)=\KC(k,\bot)=\KC(\bot,k)=\{\top\}\quad (q\in C_3),\\
&\KC(\bot,\bot)=\{\bot,\top\},\quad\KC(k,k)=\{k,\top\}.
\end{align*}
With a direct computation we deduce that $\neg$ yields homomorphisms of quantaloids
\[\neg:\HC\to\KC\quad\text{and}\quad\neg:\KC\to\HC.\]
Therefore, the negation $\neg\al$ of each $C_3$-valued similarity $\al$ is also a $C_3$-valued dissimilarity, and vice versa.
\end{exmp}

 \begin{cor} \label{frame-neg-BQCat-DQCat}
If $\sQ$ is a frame, then the assignments
\[(X,\be)\mapsto(X,\neg\be)\quad\text{and}\quad(X,\al)\mapsto(X,\neg\al)\]
define functors $\KsQ\text{-}\Cat\to\HsQ\text{-}\Cat$ and $\HsQ\text{-}\Cat\to\KsQ\text{-}\Cat$, respectively.
\end{cor}

\subsection{When $\sQ$ is a Girard quantale}

Let $m\in\sQ$. We say that
\begin{itemize}
\item $m$ is \emph{cyclic}, if $m\ldd q=q\rdd m$ for all $q\in\sQ$;
\item $m$ is \emph{dualizing}, if $(m\ldd q)\rdd m=q=m\ldd(q\rdd m)$ for all $q\in\sQ$.
\end{itemize}
It is easy to observe the following facts:
\begin{itemize}
\item If $\sQ$ is commutative, then every element of $\sQ$ is cyclic.
\item If $\sQ$ is integral, then a dualizing element of $\sQ$, whenever it exists, has to be the bottom element $\bot$ of $\sQ$.
\end{itemize}
$\sQ$ is said to be a \emph{Girard quantale} \cite{Rosenthal1990,Yetter1990} if it has a cyclic dualizing element.

\begin{exmp} \label{Girard-quantale-exmp}
For the quantales listed in Example \ref{quantale-exmp}:
\begin{enumerate}[label={\rm(\arabic*)}]
\item \label{Girard-quantale-exmp:Lawvere} Lawvere's quantale $[0,\infty]$ is not Girard.
\item \label{Girard-quantale-exmp:frame} A frame is Girard if, and only if, it is a complete Boolean algebra.
\item \label{Girard-quantale-exmp:BL} A complete BL-algebra is Girard if, and only if, it is a complete MV-algebra. In particular, the unit interval $[0,1]$ equipped with a continuous t-norm becomes a Girard quantale if, and only if, it is isomorphic to $[0,1]$ equipped with the {\L}ukasiewicz t-norm.
\item \label{Girard-quantale-exmp:nil-min} The unit interval $[0,1]$ equipped with the nilpotent minimum t-norm is Girard, in which the bottom $0$ is the only cyclic dualizing element.
\item \label{Girard-quantale-exmp:C3} $C_3$ is Girard, in which the unit $k$ is the only cyclic dualizing element (see \cite[Exercise 2.6.1]{Eklund2018}).
\item \label{Girard-quantale-exmp:Rel} The involutive quantale $\Rel(X)$ is Girard, with a cyclic dualizing element given by $X\times X-\id_X$.
\item \label{Girard-quantale-exmp:Sup} The involutive quantale $\Sup[0,1]$ is Girard, with a cyclic dualizing element given by its unit $1_{[0,1]}$, i.e., the identity map on $[0,1]$  (see \cite[Example 2.6.17(a)]{Eklund2018}).
\end{enumerate}
\end{exmp}

In a Girard quantale $\sQ$ with a cyclic dualizing element $m$, following the notation of \cite{Rosenthal1990}, we define the \emph{linear negation} of $q\in\sQ$ as
\begin{equation} \label{linear-negation-def}
q^{\perp}:=m\ldd q=q\rdd m,
\end{equation}
which clearly satisfies
\begin{equation} \label{q-bot-bot=q}
q^{\perp\perp}=q.
\end{equation}
Hence, a Girard quantale may be considered as a table of truth-values in which the law of double negation is satisfied.

\begin{rem} \label{linear-negation-vs-negation}
If a Girard quantale $\sQ$ is integral, then the linear negation coincides with the negation, i.e.,
\begin{equation} \label{integral-bot=neg}
q^{\perp}=\neg q
\end{equation}
for all $q\in\sQ$. However, Equation \eqref{integral-bot=neg} may fail in a Girard quantale whose bottom $\bot$ fails to be a cyclic dualizing element, e.g., the Girard quantales $C_3$, $\Rel(X)$ and $\Sup[0,1]$ listed in Example \ref{Girard-quantale-exmp}.
\end{rem}

\begin{lem} \label{neg-BQ-Girard}
If $\sQ$ is a Girard quantale, then the assignment
\[(b:p\lar q)\mapsto(b^{\perp}:p^{\perp}\rqa q^{\perp})\]
defines a homomorphism of quantaloids $(-)^{\perp}:\KsQ\to\HsQ$.
\end{lem}

\begin{proof}
First, $b^{\perp}\in\HsQ(p^{\perp},q^{\perp})$ if $b\in\KsQ(p,q)$. Since $\sQ$ is Girard,
\begin{align*}
b^{\perp}&=(p\ldd(b\rdd p))^{\perp} & (\text{Equation \eqref{back-diagonal-def}})\\
&=(p^{\perp\perp}\ldd(b^{\perp}\ldd p^{\perp}))^{\perp} & (b\rdd p=b^{\perp}\ldd p^{\perp})\\
&=((b^{\perp}\ldd p^{\perp})\with p^{\perp})^{\perp\perp}\\
&=(b^{\perp}\ldd p^{\perp})\with p^{\perp},
\end{align*}
and similarly $b^{\perp}=q^{\perp}\with(q^{\perp}\rdd b^{\perp})$. Hence, $b^{\perp}:p^{\perp}\rqa q^{\perp}$ is a morphism in $\HsQ$ as $b^{\perp}\leq p^{\perp}\wedge q^{\perp}$ is obvious.

Second, since $(-)^{\perp}$ preserves identities and local suprema, it remains to show that
\[c^{\perp}\diamond b^{\perp}=(c\bullet b)^{\perp}\]
for all $b:p\lar q$ and $c:q\lar r$. Indeed,
\begin{align*}
c^{\perp}\diamond b^{\perp}&=(c^{\perp}\ldd q^{\perp})\with b^{\perp} & (\text{Equation \eqref{diagonal-comp-def}})\\
&=(c\rdd q)\with b^{\perp} & (c\rdd q=c^{\perp}\ldd q^{\perp})\\
&=((c\rdd q)\with b^{\perp})^{\perp\perp}\\
&=(b^{\perp\perp}\ldd(c\rdd q))^{\perp}\\
&=(b\ldd(c\rdd q))^{\perp}\\
&=(c\bullet b)^{\perp}, & (\text{Equation \eqref{back-diagonal-comp-def}})
\end{align*}
which completes the proof.
\end{proof}

\begin{lem} \label{neg-DQ-Girard}
If $\sQ$ is a Girard quantale, then the assignment
\[(d:p\rqa q)\mapsto(d^{\perp}:p^{\perp}\lar q^{\perp})\]
defines a homomorphism of quantaloids  $(-)^{\perp}:\HsQ\to\KsQ$.
\end{lem}

\begin{proof}
First, if $d\in\HsQ(p,q)$, then $d\leq p\wedge q$ and $d=(d\ldd p)\with p=q\with(q\rdd d)$ by Equation \eqref{diagonal-def}. It follows that $d^{\perp}\geq p^{\perp}\vee q^{\perp}$ and
\[d^{\perp}=((d\ldd p)\with p)^{\perp}=p^{\perp}\ldd(d\ldd p)=p^{\perp}\ldd((p^{\perp}\ldd(d\ldd p))\rdd p^{\perp})=p^{\perp}\ldd(d^{\perp}\rdd p^{\perp}),\]
and similarly $d^{\perp}=(q^{\perp}\ldd d^{\perp})\rdd q^{\perp}$. Hence $d^{\perp}\in\KsQ(p^{\perp},q^{\perp})$ by Equation \eqref{back-diagonal-def}.

Second, since $(-)^{\perp}:\HsQ\to\KsQ$ obviously preserves identities and local suprema, it remains to check that
\[e^{\perp}\bullet d^{\perp}=(e\diamond d)^{\perp}\]
for all $d:p\rqa q$ and $e:q\rqa r$. Indeed,
\begin{align*}
e^{\perp}\bullet d^{\perp}&=d^{\perp}\ldd(e^{\perp}\rdd q^{\perp})&(\text{Equation \eqref{back-diagonal-comp-def}})\\
&=d^{\perp}\ldd((q\with e^{\perp})^{\perp})\\
&=d^{\perp}\ldd(e^{\perp\perp}\ldd q)\\
&=d^{\perp}\ldd(e\ldd q)\\
&=((e\ldd q)\with d)^{\perp}\\
&=(e\diamond d)^{\perp}, &(\text{Equation \eqref{diagonal-comp-def}})
\end{align*}
which completes the proof.
\end{proof}

The homomorphisms of quantaloids given in Lemmas \ref{neg-BQ-Girard} and \ref{neg-DQ-Girard} are obviously inverse to each other by Equation \eqref{q-bot-bot=q}, and thus they are both isomorphisms between the quantaloids $\HsQ$ and $\KsQ$. Moreover, it is clear that both of them can be extended to isomorphisms of quantaloids between $\DsQ$ and $\BsQ$, and therefore:

\begin{thm} \label{Girard-iso}
If $\sQ$ is a Girard quantale, then there are isomorphisms
\[\DsQ\cong\BsQ\quad\text{and}\quad\HsQ\cong\KsQ\]
of quantaloids, and consequently,  the assignment $(X,\al)\mapsto(X,\al^{\perp})$ defines an isomorphism
\[\HsQ\text{-}\Cat\cong\KsQ\text{-}\Cat\]
of categories. 
\end{thm}

Note that for each of the involutive Girard quantales listed in Example \ref{Girard-quantale-exmp}, the cyclic dualizing element given there is hermitian. Actually, whenever $\sQ$ is an involutive Girard quantale with a hermitian and cyclic dualizing element, it is easy to verify that the homomorphisms of quantaloids given in Lemmas \ref{neg-BQ-Girard} and \ref{neg-DQ-Girard} both preserve the involution of $\sQ$, and in this case:

\begin{thm} \label{Girard-sim-dissim} \footnote{The authors are grateful to an anonymous referee for helpful remarks on this theorem.}
If $\sQ$ is an involutive Girard quantale with a hermitian and cyclic dualizing element, then $\sQ$-valued similarities and $\sQ$-valued dissimilarities are interdefinable by the aid of linear negation; that is, the linear negation $\al^{\perp}$ of each $\sQ$-valued similarity $\al$ is a $\sQ$-valued dissimilarity, and conversely, the linear negation $\be^{\perp}$ of each $\sQ$-valued dissimilarity $\be$ is a $\sQ$-valued similarity.
\end{thm}

\begin{rem} \label{sim-vs-apart-Boolean}
If $\sB$ is a complete Boolean algebra, from Theorem \ref{Girard-iso} we know that $\sB$-valued similarities and $\sB$-valued dissimilarities are interdefinable by passing to complements. Note that
\[\al(p,q):=p\wedge q\]
for all $p,q\in\sB$ defines a $\sB$-valued similarity $\al$ on $\sB$ itself, but its negation cannot be made into a $\sB$-valued apartness relation on $\sB$ (see Remark \ref{dissim-vs-apart}), because
\[\neg\al(q,q)=\neg q\neq\bot\]
as long as $q\neq\top$. So, even in the Boolean-valued case, similarities and apartness relations are not interdefinable by passing to complements.  However, as pointed out to us by an anonymous referee, there is also a natural way to switch between similarities and apartness relations in the Boolean-valued case.

If $(X,E,\gamma)$ is a $\sB$-valued apartness relation, then
\[\al(x,y):=E(x)\wedge E(y)\wedge \neg\gamma(x,y)\]
defines a $\sB$-valued similarity on $X$.  Conversely, if $\al$ is a $\sB$-valued similarity on $X$, then $(X,E,\gamma)$ is a $\sB$-valued apartness relation, where $E(x):=\al(x,x)$ and \[\gamma(x,y):=E(x)\wedge E(y)\wedge \neg\al(x,y)\]
for all $x,y\in X$. We note en passe that the principal lower set $\da\!(E(x)\wedge E(y))$ of $\sB$ is itself a complete Boolean algebra and $\al(x,y)$ is the complement of $\gamma(x,y)$ in this Boolean algebra ($\gamma(x,y)$ belongs to the Boolean algebra because $\gamma(x,y)\leq E(x)\wedge E(y)$).
\end{rem}

\begin{exmp} \label{nil-min-neg}
As an immediate consequence of Theorem \ref{Girard-iso}, we may find another example for the non-necessity of the assumptions on $\sQ$ in Propositions \ref{divisible-neg-dissim} and \ref{frame-neg-sim-dissim}. Let $\sQ$ be the unit interval $[0,1]$ equipped with the nilpotent minimum t-norm (see Example \ref{quantale-exmp}\ref{quantale-exmp:nil-min}). Then $\sQ$ is not a divisible quantale, hence not a frame, but the negation operator
\[\neg:\HsQ\to\KsQ\]
is an isomorphism of quantaloids since $\sQ$ is a Girard quantale with the bottom $0$ being the cyclic dualizing element (see Example \ref{Girard-quantale-exmp}\ref{Girard-quantale-exmp:nil-min}).
\end{exmp}

\begin{rem}
Since the quantale $C_3$ is Girard (see Example \ref{Girard-quantale-exmp}\ref{Girard-quantale-exmp:C3}), by applying the linear negation \eqref{linear-negation-def} we are able to switch between $C_3$-valued similarities and $C_3$-valued dissimilarities. It is interesting that for this quantale, as Example \ref{C3-neg} shows, the negation \eqref{negation-commutative-def} also makes sense while considering the interactions between similarities and dissimilarities.
\end{rem}

In Theorem \ref{Girard-iso}, the interdefinability of $\sQ$-valued similarities and $\sQ$-valued dissimilarities follows from the isomorphism
\[\HsQ\cong\KsQ\]
when $\sQ$ is Girard. It is now natural to ask whether $\sQ$ being Girard is essential for establishing the isomorphism $\HsQ\cong\KsQ$. In what follows we are able to provide an affirmative answer for a commutative and integral quantale $\sQ$ (see Corollary \ref{DQ-BQ-Girard-integral}). Actually, we have the following:

\begin{thm} \label{DQ-BQ-Girard}
Let $\sQ$ be a commutative quantale. Then there is an isomorphism
\[\DsQ\cong\BsQ\]
of quantaloids if, and only if, $\sQ$ is a Girard quantale.
\end{thm}

As a preparation, let us investigate properties of the quantale
\[\BsQ(q,q)\]
for a given quantale $\sQ=(\sQ,\with,k)$ and a \emph{cyclic} element $q\in\sQ$. Since $\BsQ(q,q)$ is equipped with the \emph{reverse} order inherited from $\sQ$, in order to eliminate ambiguity we use they symbol ``$\preceq$'' for the order in $\BsQ(q,q)$; that is,
\[b\preceq b'\ \text{in}\ \BsQ(q,q)\iff b'\leq b\ \text{in}\ \sQ.\]
Moreover, we denote by $\lda$, $\rda$ the implications in $\BsQ$, and reserve $\ldd$, $\rdd$ for implications in $\sQ$.

\begin{lem} \label{BQq-cyclic}
If $q\in\sQ$ is a cyclic element, then
\[b'\lda b=q\ldd(b'\rdd b)\quad\text{and}\quad b\rda b'=(b\ldd b')\rdd q\]
for all $b,b'\in\BsQ(q,q)$.
\end{lem}

\begin{proof}
It is clear that $q\ldd(b'\rdd b)\in\BsQ(q,q)$ by Equation \eqref{back-diagonal-def}, and
\begin{align*}
b''\bullet b\preceq b'&\iff b'\leq b''\bullet b=(q\ldd b)\rdd b''&(\text{Equation \eqref{back-diagonal-comp-def}})\\
&\iff (q\ldd b)\with b'\leq b''\\
&\iff q\ldd(((q\ldd b)\with b')\rdd q)\leq b''&(b''=q\ldd(b''\rdd q))\\
&\iff q\ldd(b'\rdd((q\ldd b)\rdd q))\leq b''\\
&\iff q\ldd(b'\rdd b)\leq b''&(b=(q\ldd b)\rdd q)\\
&\iff b''\preceq q\ldd(b'\rdd b)
\end{align*}
for all $b''\in\BsQ(q,q)$. Thus $b'\lda b=q\ldd(b'\rdd b)$. Similarly we obtain $b\rda b'=(b\ldd b')\rdd q$.
\end{proof}

\begin{prop} \label{BQq-Girard}
If $q\in\sQ$ is a cyclic element, then $\BsQ(q,q)$ is a Girard quantale. In particular, if $\sQ$ is a commutative quantale, then $\BsQ(q,q)$ is a Girard quantale for all $q\in\sQ$.
\end{prop}

\begin{proof}
We show that $m_q:=q\ldd q=q\rdd q$ is a cyclic dualizing element of the quantale $\BsQ(q,q)$.

First, $m_q$ is cyclic. For any $b\in\BsQ(q,q)$, the conjunction of
\begin{align*}
b\ldd m_q&=b\ldd(q\rdd q)\\
&=((q\ldd b)\rdd q)\ldd(q\rdd q)&(b=(q\ldd b)\rdd q) \\
&=(q\ldd b)\rdd(q\ldd(q\rdd q))\\
&=(q\ldd b)\rdd q\\
&=b \\
&=q\ldd(b\rdd q) \\
&=((q\ldd q)\rdd q)\ldd(b\rdd q)\\
&=(q\ldd q)\rdd(q\ldd(b\rdd q))\\
&=(q\ldd q)\rdd b &(b=q\ldd(b\rdd q))\\
&=m_q\rdd b
\end{align*}
and Lemma \ref{BQq-cyclic} yields that
\begin{equation} \label{b-imp-mq}
m_q\lda b=q\ldd(m_q\rdd b)=q\ldd b=b\rdd q=(b\ldd m_q)\rdd q=b\rda m_q.
\end{equation}

Second, $m_q$ is dualizing. Since for any $b\in\BsQ(q,q)$,
\begin{align*}
m_q\lda(b\rda m_q)&=q\ldd(b\rdd q)&(\text{Equation \eqref{b-imp-mq}})\\
&=(q\ldd b)\rdd q \\
&=(m_q\lda b)\rda m_q,&(\text{Equation \eqref{b-imp-mq}})
\end{align*}
the conclusion thus follows.
\end{proof}

\begin{rem}
For a cyclic element $q$ of a quantale $\sQ$, it is shown in \cite{Rosenthal1990a} that
\[\sj=((-)\rdd q)\rdd q:\sQ\to\sQ\]
is a \emph{nucleus} \cite{Rosenthal1990} on $\sQ$, and the resulting quotient quantale
\[\sQ_{\sj}=(\sQ_{\sj},\with_{\sj},q\rdd q)\]
is a Girard quantale (see \cite[Theorem 3.1.1]{Rosenthal1990} for the construction of $\sQ_{\sj}$). As we will see below, the Girard quantale $\BsQ(q,q)$ obtained in Proposition \ref{BQq-Girard} is isomorphic to $\sQ_{\sj}$.

Let $\sR=(\sR,\with,k)$ be a Girard quantale with a cyclic dualizing element $m$. Then from $\sR$ we may construct another Girard quantale $\sR^d=(\sR^d,\with^d,m)$, which is isomorphic to $\sR$ with the correspondence $q\mapsto q^{\perp}$ being an isomorphism of quantales:
\begin{itemize}
\item elements of $\sR^d$ are the same as those of $\sR$, and $\sR^d$ is equipped with the \emph{reversed} order of $\sR$;
\item the multiplication on $\sR^d$ is defined by
    \begin{equation} \label{Girard-dual-comp}
    p\with^d q=(p^{\perp}\with q^{\perp})^{\perp}=q\ldd p^{\perp}=q^{\perp}\rdd p
    \end{equation}
    for all $p,q\in\sR^d$;
\item the unit of $\sR^d$ is the cyclic dualizing element $m$ of $\sR$, and the unit $k$ of $\sR$ is a cyclic dualizing element of $\sR^d$.
\end{itemize}

Now we show that $\sQ_{\sj}$ is isomorphic to $(\BsQ(q,q))^d$, and hence to $\BsQ(q,q)$. To see this, just note that the underlying sets of $\sQ_{\sj}$ and $\BsQ(q,q)$ are the same, and
\begin{align*}
b\with_{\sj}c&=\sj(b\with c)&\text{(see \cite[Theorem 3.1.1]{Rosenthal1990})}\\
&=((b\with c)\rdd q)\rdd q\\
&=(c\rdd(b\rdd q))\rdd q\\
&=q\ldd(c\rdd(b\rdd q))&(q\ \text{is cyclic})\\
&=c\lda(b\rdd q)&(\text{Lemma \ref{BQq-cyclic}})\\
&=c\lda(m_q\lda b)&(\text{Equation \eqref{b-imp-mq}})\\
&=b\bullet^d c&(\text{Equation \eqref{Girard-dual-comp}})
\end{align*}
for all $b,c\in\sQ_{\sj}$.
\end{rem}

Note that every $q\in\sQ$ satisfies
\[(q\ldd k)\with k=q=k\with(k\rdd q);\]
that is, $q\in\DsQ(k,k)$ for all $q\in\sQ$. Moreover,
\[p\diamond q=(p\ldd k)\with q=p\with(k\rdd q)=p\with q\]
for all $p,q\in\DsQ(k,k)$. Hence, $\sQ$ and $\DsQ(k,k)$ are the same quantales, upon which the proof of Theorem \ref{DQ-BQ-Girard} is obtained:

\begin{proof}[Proof of Theorem \ref{DQ-BQ-Girard}]
The ``if'' part is already obtained in Theorem \ref{Girard-iso}. For the ``only if'' part, note that
\[\sQ=\DsQ(k,k).\]
Hence, the isomorphism $\BsQ\cong\DsQ$ guarantees that $\sQ\cong\BsQ(q,q)$ for some $q\in\sQ$. Since $\sQ$ is commutative, Proposition \ref{BQq-Girard} ensures that $\BsQ(q,q)$ is a Girard quantale, and therefore so is $\sQ$.
\end{proof}

Since $\DsQ\cong\HsQ$ and $\BsQ\cong\KsQ$ when $\sQ$ is integral, the following corollary is an immediate consequence of Theorem \ref{DQ-BQ-Girard}:

\begin{cor} \label{DQ-BQ-Girard-integral}
Let $\sQ$ be a commutative and integral quantale. Then there is an isomorphism
\[\HsQ\cong\KsQ\]
of quantaloids if, and only if, $\sQ$ is a Girard quantale.
\end{cor}

\section*{Acknowledgement}

The first, the second and the fourth named authors acknowledge the support of National Natural Science Foundation of China (No. 11771310, No. 11701396 and No. 11871358).

The authors thank the referees gratefully for their valuable comments and suggestions which help improve the paper significantly.





\end{document}